\documentclass[letterpaper,11pt]{amsart}

\usepackage[margin=1.2in]{geometry}
\usepackage{amsmath,amsthm,amssymb}
\usepackage{xspace,xcolor}
\usepackage[breaklinks,colorlinks,citecolor=teal,linkcolor=teal,urlcolor=teal,pagebackref,hyperindex]{hyperref}
\usepackage[alphabetic]{amsrefs}
\usepackage[all]{xy}
\usepackage{color}
\usepackage{comment} 

\usepackage{tikz-cd}

\theoremstyle{plain}
\newtheorem{thm}{Theorem}[section]

\newtheorem{lem}[thm]{Lemma}
\newtheorem{prop}[thm]{Proposition}

\theoremstyle{definition}

\theoremstyle{remark}
\newtheorem{rmk}[thm]{Remark}

\def\cO{\mathcal{O}}

\def\fra{\mathfrak{a}}
\def\frb{\mathfrak{b}}

\def\.{\cdot}
\def\^{\widehat}

\def\({\left(}
\def\){\right)}

\renewcommand{\and}{ \ \ \text{ and } \ \ }

\makeatletter
\@namedef{subjclassname@2020}{\textup{2020} Mathematics Subject Classification}
\makeatother

\usepackage{soul}

\begin{document}

\begin{abstract}
This note points out a gap in the proof of one of the technical results in the paper \emph{Asymptotic Invariants of Base Loci}, that appeared 
in \emph{Ann. Inst. Fourier (Grenoble)} 56 (2006), 1701--1734. We provide a correct proof of this result.
\end{abstract}

\title{ Erratum to the paper: Asymptotic Invariants of Base Loci}

\author[L.~Ein]{Lawrence Ein}

\address{Department of Mathematics, University of Illinois at Chicago, 851 South Morgan St., Chicago, IL 60607, USA}

\email{ein@uic.edu}

\author[R.~Lazarsfeld]{Robert Lazarsfeld}

\address{Department of Mathematics, Stony Brook University, Stony Brook, NY 11794, USA}

\email{robert.lazarsfeld@stonybrook.edu}

\author[M.~Musta\c{t}\u{a}]{Mircea Musta\c{t}\u{a}}

\address{Department of Mathematics, University of Michigan, 530 Church Street, Ann Arbor, MI 48109, USA}

\email{mmustata@umich.edu}

\author[M.~Nakamaye]{Michael Nakamaye}

\address{3900 Moon St. NE, Albuquerque, NM 87111, USA}

\email{nakamaye@gmail.com}

\author[M.~Popa]{Mihnea~Popa}
\address{Department of Mathematics, Harvard University, 
1 Oxford Street, Cambridge, MA 02138, USA} 
\email{mpopa@math.harvard.edu}

\maketitle

\section{The setup}

We work over an algebraically closed field $k$ and let $X$ be a variety over $k$ (that is, a scheme of finite type over $k$ that is irreducible and reduced). 
Let $N$ be a finitely generated, free abelian group and
$S\subseteq N$ a finitely generated, saturated subsemigroup. We denote by $C$ the cone generated by $S$ in $N_{\mathbf R}=N\otimes_{\mathbf Z}{\mathbf R}$,
so $C$ is a rational polyhedral convex cone and $S=C\cap N$. For standard facts of convex geometry, we refer to \cite{Ewald} and \cite{Ziegler}. 

An \emph{$S$-graded system of ideals} on $X$ is a family $\fra_{\bullet}=(\fra_m)_{m\in S}$ 
of coherent ideals $\fra_m\subseteq\cO_X$ for $m\in S$ such that $\fra_0=\cO_X$ and $\fra_{m}\cdot\fra_{m'}\subseteq\fra_{m+m'}$ for every $m,m'\in S$.
The \emph{Rees algebra} of $\fra_{\bullet}$ is the quasi-coherent sheaf of $S$-graded $\cO_X$-algebras
$$R(\fra_{\bullet})=\bigoplus_{m\in S}\fra_m.$$
We say that $\fra_{\bullet}$ is \emph{finitely generated} if $R(\fra_{\bullet})$ is a finitely generated $\cO_X$-algebra. 
For a coherent ideal $\frb$ on $X$, we denote by $\overline{\frb}$ the integral closure of $\frb$ (see \cite{HS} for the definition
and basic properties of integral closure of ideals).

The following result is Proposition~4.7 in \cite{ELMNP}. 

\begin{prop}\label{main_result}
If $\fra_{\bullet}$ is a finitely generated $S$-graded system of ideals on the variety $X$, then there is a smooth fan $\Delta$ with support $C$
such that for every smooth fan $\Delta'$ refining $\Delta$, there is a positive integer $d$ with the following property: if $\sigma$ is a cone in $\Delta'$ and
$e_1,\ldots,e_s$ are the generators of
$S_{\sigma}=\sigma\cap N$, then
\begin{equation}\label{eq_main_result}
\overline{\fra_{d\sum_ip_ie_i}}=\overline{\prod_i\fra_{de_i}^{p_i}}\quad\text{for all}\quad p_1,\ldots,p_s\in {\mathbf Z}_{\geq 0}.
\end{equation}
\end{prop}

The argument in \cite{ELMNP} proceeds by induction on the dimension of $C$. A key claim is that one can choose a smooth fan $\Delta$ with support $C$
such that the degrees corresponding to a finite system of generators of $R(\fra_{\bullet})$ lie on the rays of $\Delta$ and such that the equality (\ref{eq_main_result})
holds on each cone of $\Delta$ of dimension $\dim(C)-1$. However, it is not clear that this can be achieved when $\dim(C)\geq 3$: given any fan $\Delta$
with support $C$, we can apply the inductive hypothesis to get suitable refinements for the cones in $\Delta$ of dimension $\dim(C)-1$, but we then need to further refine $\Delta$, leading to new cones
of dimension $\dim(C)-1$. It is not clear that this process terminates.

\section{The corrected proof}

In what follows we provide a different proof of Proposition~\ref{main_result}. The key ingredient is the following general lemma. While the statement is familiar to the 
experts in convex geometry, we provide a proof since we could not find a reference in the literature.

\begin{lem}\label{main_lemma}
Let $N$ be a finitely generated, free abelian group and $C$ the convex cone in $N_{\mathbf R}$ generated by
$v_1,\ldots,v_r\in N_{\mathbf Q}=N\otimes_{\mathbf Z}{\mathbf Q}$. Given $\alpha=(\alpha_1,\ldots,\alpha_r)\in {\mathbf R}_{\geq 0}^r$, we consider the function
$\varphi_{\alpha}\colon C\cap N_{\mathbf Q}\to {\mathbf R}_{\geq 0}$ given by
\begin{equation}\label{eq_phi}
\varphi_{\alpha}(v)=\inf\big\{\lambda_1\alpha_1+\ldots+\lambda_r\alpha_r\mid \lambda_1,\ldots,\lambda_r\in {\mathbf Q}_{\geq 0}, \lambda_1v_1+\ldots+\lambda_rv_r=v\big\}.
\end{equation}
For every $\alpha$, the infimum in (\ref{eq_phi}) is a minimum and $\varphi_{\alpha}$ is a convex, piecewise linear function. Moreover, there is a fan $\Delta$,
with support $C$, such that each $\varphi_{\alpha}$ is linear on every cone of $\Delta$.
\end{lem}

Before giving the proof of the lemma, we recall one well-known fact. For $u=(u_1,\ldots,u_n),v=(v_1,\ldots,v_n)\in {\mathbf R}^n$, we put
$\langle u,v\rangle=\sum_{i=1}^nu_iv_i$. We use the same notation for the corresponding pairing of vectors in ${\mathbf R}^r$.

\begin{rmk}\label{rmk1}
Recall that a (rational) \emph{polyhedron} in ${\mathbf R}^n$ is a subset defined by finitely many affine linear inequalities (defined over ${\mathbf Q}$).
A (rational) \emph{polytope} in ${\mathbf R}^n$ is a bounded (rational) polyhedron, or equivalently, the convex hull of finitely many points (in ${\mathbf Q}^n$); 
see \cite[Theorem~1.1]{Ziegler}. Any polyhedron $P$ in ${\mathbf R}^n$ can be written as $P_0+C$, where $P_0$ is a polytope and $C$ is a polyhedral convex cone
(see \cite[Theorem~1.2]{Ziegler}); moreover, if $P$ is rational, then $P_0$ and $C$ can be taken rational as well. Suppose now that $\ell$ is a linear function on
${\mathbf R}^n$ given by $\ell(v)=\langle u,v\rangle$ for some $u\in {\mathbf R}^n$. It is clear that $\ell$ is bounded below on $P$ if and only if $\ell\geq 0$ on $C$,
in which case we have
$$\inf_{v\in P}\ell(v)=\min_{v\in P}\ell(v)=\min_{v\in P_0}\ell(v)=\min\big\{\ell(w_1),\ldots,\ell(w_s)\big\},$$
where $w_1,\ldots,w_s$ are the vertices (that is, the $0$-dimensional faces) of $P_0$. Note that if $P_0$ is a rational polytope, then $w_i\in {\mathbf Q}^n$ for all $i$,
hence if $P$ is a rational polyhedron, we have
$$\min_{v\in P}\ell(v)=\min_{v\in P\cap {\mathbf Q}^n}\ell(v).$$
\end{rmk}

\begin{proof}[Proof of Lemma~\ref{main_lemma}]
Let us choose an isomorphism
$N\simeq {\mathbf Z}^n$ that allows us to identify $N_{\mathbf Q}$ and $N_{\mathbf R}$ with ${\mathbf Q}^n$ and ${\mathbf R}^n$, respectively. We can thus
write $v_i=(v_{i,1},\ldots,v_{i,n})$ for $1\leq i\leq r$, with $v_{i,j}\in {\mathbf Q}$ for all $i$ and $j$. 
For $\alpha\in {\mathbf R}_{\geq 0}^r$, 
let us denote by $\widetilde{\varphi}_{\alpha}$ the map $C\to {\mathbf R}$ given by 
\begin{equation}\label{eq_phi2}
\widetilde{\varphi}_{\alpha}(v)=\inf\big\{\langle \alpha,\lambda\rangle \mid \lambda=(\lambda_1,\ldots,\lambda_r)\in {\mathbf R}^r_{\geq 0}, \lambda_1v_1+\ldots+\lambda_rv_r=v\big\}.
\end{equation}

If $v=(b_1,\ldots,b_n)\in {\mathbf R}^n$ and $\lambda=(\lambda_1,\ldots,\lambda_r)\in {\mathbf R}^r$, the conditions
$\lambda_1,\ldots,\lambda_r\geq 0$ and  $v=\sum_{i=1}^r\lambda_iv_i$ are equivalent to 
$\lambda\in P(v)$, where $P(v)$ is the polyhedron in ${\mathbf R}^r$ given by
$$\sum_{i=1}^rv_{i,j}\lambda_i=b_j\,\,\text{for}\,\,1\leq j\leq n\quad\text{and}\quad \lambda_i\geq 0\,\,\text{for}\,\,1\leq i\leq r.$$
For every $\alpha=(\alpha_1,\ldots,\alpha_r)\in {\mathbf R}^r_{\geq 0}$, we have $\langle \alpha,\lambda\rangle\geq 0$ for all $\lambda\in P(v)$. 
We thus conclude using Remark~\ref{rmk1} that $\varphi_{\alpha}(v)=\widetilde{\varphi}_{\alpha}(v)$ for every $v\in C\cap {\mathbf Q}^n$ and the infimum in the definition of
$\varphi_{\alpha}(v)$ is a minimum.

The fact that each $\widetilde{\varphi}_{\alpha}$ is a convex function follows easily from the definition. Indeed, since 
we clearly have $\widetilde{\varphi}_{\alpha}(tv)=t\cdot \widetilde{\varphi}_{\alpha}(v)$ for all $v\in C$ and $t\geq 0$, convexity is equivalent to the fact that
$$\widetilde{\varphi}_{\alpha}(v+v')\leq \widetilde{\varphi}_{\alpha}(v)+\widetilde{\varphi}_{\alpha}(v')\quad\text{for all}\quad v,v'\in C.$$
This follows from the fact that if  $v=\sum_{i=1}^r\lambda_iv_i$ and $v'=\sum_{i=1}^r\lambda'_iv_i$, with $\lambda_i,\lambda'_i\in {\mathbf R}_{\geq 0}$ for all $i$,
are such that $\sum_{i=1}^r\lambda_i\alpha_i=\widetilde{\varphi}_{\alpha}(v)$ and $\sum_{i=1}^r\lambda'_i\alpha_i=\widetilde{\varphi}_{\alpha}(v')$,
then $v+v'=\sum_{i=1}^r(\lambda_i+\lambda'_i)v_i$, hence
$$\widetilde{\varphi}_{\alpha}(v+v')\leq \sum_{i=1}^r\lambda_i\alpha_i+\sum_{i=1}^r\lambda'_i\alpha_i=\widetilde{\varphi}_{\alpha}(v)+\widetilde{\varphi}_{\alpha}(v').$$

It is a consequence of the Duality Theorem in Linear Programming (see \cite[Theorem~IV.8.2]{Barvinok}) that for every $v=(b_1,\ldots,b_n)\in C$, we have
$$\widetilde{\varphi}_{\alpha}(v)=\max_{\gamma\in Q(\alpha)}\langle v,\gamma\rangle,$$
where $Q(\alpha)$ is the polyhedron in ${\mathbf R}^n$ consisting of those $\gamma=(\gamma_1,\ldots,\gamma_n)$ with
$\sum_{j=1}^nv_{i,j}\gamma_j\leq \alpha_i$ for $1\leq i\leq r$. 

In order to complete the proof of the lemma, it is enough to show that there is a fan $\Delta$ (consisting of strongly convex, rational polyhedral convex cones) with support
$C$, such that $\widetilde{\varphi}_{\alpha}$ is a linear function on every cone of $\Delta$ for all $\alpha\in {\mathbf R}_{\geq 0}^n$. Note that if $v\in C$ is fixed
and we write $P(v)$ as $P_0(v)+C_0(v)$, for a polytope $P_0(v)$ and a polyhedral convex cone $C_0(v)$, then it follows from Remark~\ref{rmk1} that
$$\widetilde{\varphi}_{\alpha}(v)=\min\big\{\langle \alpha,w_1\rangle,\ldots,\langle \alpha,w_s\rangle\big\},$$
where $w_1,\ldots,w_s$ are the vertices of $P_0(v)$. We thus conclude that the function $\alpha\mapsto\widetilde{\varphi}_{\alpha}(v)$ is continuous
on ${\mathbf R}_{\geq 0}^r$. In particular, it is enough to find a fan $\Delta$ as above such that $\widetilde{\varphi}_{\alpha}$ is linear on the cones
of $\Delta$ for all $\alpha\in {\mathbf R}_{>0}^r$. 

Note now that if $\alpha\in {\mathbf R}_{>0}^r$, then $0$ lies in the interior of $Q(\alpha)$. We consider the normal fan\footnote{We consider the version of the normal fan whose rays are the outer normals to the facets of the polyhedron.} $\Delta(\alpha)$ to $Q(\alpha)$
(see \cite[Example~7.3]{Ziegler}). Its cones are of the form 
$$\sigma_F=\big\{w\in {\mathbf R}^n\mid \langle w,u'\rangle \geq \langle w,u\rangle\,\,\text{for all}\,\,u\in Q(\alpha),u'\in F\big\},$$
where $F$ runs over the faces of $Q(\alpha)$. It is clear that $\widetilde{\varphi}_{\alpha}$ is linear on each cone of $\Delta(\alpha)$: on $\sigma_F$ it is given by
$\langle -,u'\rangle$ for every $u'\in F$.

The support of $\Delta(\alpha)$ consists precisely of those $w\in {\mathbf R}^n$
such that the function $\langle w,-\rangle$ is bounded above on $Q(\alpha)$; equivalently, if we write $Q(\alpha)=Q_0(\alpha)+T(\alpha)$, where $Q_0(\alpha)$ is a polytope
and $T(\alpha)$ is a polyhedral convex cone, then $-w$ lies in the dual $T(\alpha)^{\vee}$ of $T(\alpha)$. Note that by definition of $Q(\alpha)$, the cone $T(\alpha)$ is defined by
$\sum_{j=1}^nv_{i,j}\gamma_j\leq 0$ for $1\leq i\leq r$ (see \cite[Proposition~1.12]{Ziegler}), hence it is the dual of $-C$. We thus conclude that the support of $\Delta(\alpha)$ is $C$. 

Note also that every facet of $Q(\alpha)$ is of the form 
$$Q(\alpha)\cap \big\{\gamma\mid \langle v_i,\gamma\rangle=\alpha_i\big\}$$
for some (nonzero) $v_i$, hence the corresponding ray of $\Delta(\alpha)$ is ${\mathbf R}_{\geq 0}v_i$. We deduce that when we vary $\alpha$,
the rays of $\Delta(\alpha)$ belong to a finite set, hence we have finitely many such fans. If we let $\Delta$ be any common refinement of all such $\Delta(\alpha)$,
we conclude that the support of $\Delta$ is $C$ and 
$\widetilde{\varphi}_{\alpha}$ is linear on the cones of $\Delta$ for all $\alpha\in {\mathbf R}_{>0}^r$. This completes the proof of the lemma.
\end{proof}

\begin{rmk}
We now explain a more elementary argument for the existence of the fan $\Delta$ in Lemma~\ref{main_lemma}. This avoids the use of the Duality Theorem in Linear Programming and
also makes the choice of fan $\Delta$ more explicit.
First, arguing as in the proof of  Carath\'{e}odory's theorem (see \cite[Proposition~1.15]{Ziegler}), we show the following

\noindent {\bf Claim}: in the definition of $\widetilde{\varphi}_{\alpha}(v)$ it is enough
to only consider those $\lambda=(\lambda_1,\ldots,\lambda_r)\in {\mathbf R}_{\geq 0}^r$ with the property that the $v_i$ with $i\in J(\lambda):=\{i\mid\lambda_i\neq 0\}$ are linearly independent. 

In order to see this, it is enough to show that if the $v_i$ with $i\in J(\lambda)$ are
linearly dependent, then we can find $\lambda'=(\lambda'_1,\ldots,\lambda'_r)\in {\mathbf R}_{\geq 0}^r$ such that $\sum_{i=1}^r\lambda_iv_i=\sum_{i=1}^r\lambda'_iv_i$
and we have $\sum_{i=1}^r\lambda'_i\alpha_i\leq \sum_{i=1}^r\lambda_i\alpha_i$ and $J(\lambda')\subsetneq J(\lambda)$. Note that, by assumption, we have a relation
$\sum_{i\in J(\lambda)}b_i\lambda_i=0$ such that $J:=\{i\in J(\lambda)\mid b_i\neq 0\}$ is nonempty. After possibly multiplying this relation with $-1$, we may and will assume
that $\sum_{i\in J}b_i\alpha_i\geq 0$ and $b_i>0$ for some $i\in J$ (we use here the fact that $\alpha_i\geq 0$ for all $i$). Let $j\in J$ be such that
$$\tfrac{\lambda_j}{b_j}=\min\left\{\tfrac{\lambda_i}{b_i}\mid i\in J, b_i>0\right\}.$$
In this case, it is straightforward to see that if $\lambda'_i=\lambda_i-\tfrac{\lambda_j}{b_j}b_i$ for all $i\in J$ and $\lambda'_i=0$ for $i\not\in J$,
then $\lambda'\in {\mathbf R}_{\geq 0}^r$ and we have $J(\lambda')\subseteq J(\lambda)\smallsetminus\{j\}$ and 
$$v=\sum_{i=1}^r\lambda'_iv_i\quad\text{and}\quad \sum_{i=1}^r\lambda'_i\alpha_i\leq \sum_{i=1}^r\lambda_i\alpha_i.$$
This proves the claim.

Let $\Lambda$ be the set of those $J\subseteq \{1,\ldots,r\}$ such that the $v_i$ with $i\in J$ are linearly independent. For every $J\in\Lambda$,
let $\sigma_J$ be the convex cone in $N_{\mathbf R}$
generated by the $v_i$ with $i\in J$. 
It is a consequence of
Carath\'{e}odory's theorem that $C=\bigcup_{J\in\Lambda}\sigma_J$. Consider a fan $\Delta$ with support $C$ such that every cone $\sigma_J$, for $J\in\Lambda$, is a union
of cones in $\Delta$. We now show that for every $\alpha\in {\mathbf R}^r_{\geq 0}$ and every $\tau\in\Delta$, the restriction $\widetilde{\varphi}_{\alpha}\vert_{\tau}$ is a linear function. 
Note first that if $J\in\Lambda$ and for some $v\in \sigma_J$ we write $v=\sum_{i\in J}\lambda_iv_i$, then each $\lambda_i$ is given by a linear function of $v$; therefore
$\sum_{i\in J}\lambda_i\alpha_i$ is given by a linear function $\ell_J$ of $v$. We next note that if $w$ lies in the relative interior ${\rm Relint}(\tau)$ of $\tau$
and $J\in\Lambda$, then $w\in\sigma_J$ if and only if $\tau\subseteq\sigma_J$. Indeed, by construction of $\Delta$, we have $\sigma_1,\ldots,\sigma_d\in \Delta$ such that
$\sigma_J=\bigcup_{j=1}^d\sigma_j$, hence $\sigma_J\cap\tau=\bigcup_{j=1}^d(\sigma_j\cap \tau)$. Since $\Delta$ is a fan, each $\sigma_j\cap \tau$ is a face 
of $\tau$, so the union contains a point in ${\rm Relint}(\tau)$ if and only if $\tau\subseteq\sigma_j$ for some $j$, in which case $\tau\subseteq\sigma_J$. 
Our claim thus implies that 
\begin{equation}\label{eq_relint}
\widetilde{\varphi}_{\alpha}(v)=\min\big\{\ell_J(v)\mid \tau\subseteq \sigma_J\big\}\quad\text{for all}\quad v\in {\rm Relint}(\tau).
\end{equation}
It is well-known (and easy to see) that (\ref{eq_relint}) implies that $-\widetilde{\varphi}_{\alpha}$ is convex on ${\rm Relint}(\tau)$. Since $\widetilde{\varphi}_{\alpha}$ is a convex function on $C$
(the easy argument for this was given in the proof of Lemma~\ref{main_lemma}), it follows that we have a linear function $\ell$ on $N_{\mathbf R}$ such that
$\widetilde{\varphi}_{\alpha}=\ell$ on ${\rm Relint}(\tau)$. Given any $v\in\tau$, if $v'\in {\rm Relint}(\tau)$, then $v+v'\in {\rm Relint}(\tau)$, and the convexity of $\widetilde{\varphi}_{\alpha}$
implies that
$$\ell(v+v')=\widetilde{\varphi}_{\alpha}(v+v')\leq\widetilde{\varphi}_{\alpha}(v)+\widetilde{\varphi}_{\alpha}(v')=\widetilde{\varphi}_{\alpha}(v)+\ell(v').$$
Therefore we have $\ell\leq\widetilde{\varphi}_{\alpha}$ on $\tau$. On the other hand, it is an immediate consequence of the definition of $\widetilde{\varphi}_{\alpha}$ that if 
$(w_m)_{m\geq 1}$ is a sequence of vectors in $C$ with $\lim_{m\to\infty}w_m=v$, then
$\widetilde{\varphi}_{\alpha}(v)\leq \liminf_{m\to\infty}\widetilde{\varphi}_{\alpha}(w_m)$. By taking $w_m\in {\rm Relint}(\tau)$, we see that $\widetilde{\varphi}_{\alpha}\leq \ell$ on $\tau$.
We thus conclude that $\widetilde{\varphi}_{\alpha}\vert_{\tau}$ is a linear function.
\end{rmk}

We can now give the proof of the result from \cite{ELMNP}.

\begin{proof}[Proof of Proposition~\ref{main_result}]
Recall first that every fan admits a smooth refinement (which has the same support), see \cite[Theorem~8.5]{Ewald}. Furthermore, it is clear that if $\Delta$ is a smooth
fan whose cones satisfy (\ref{eq_main_result}), then any smooth refinement of $\Delta$ satisfies the same property. 

Let 
$$T=\{m\in S\mid \fra_m\neq 0\}\quad\text{and}\quad S_+=\{m\in S\mid \ell m\in T\,\,\text{for some}\,\,\ell\in {\mathbf Z}_{>0}\}.$$
Since $R(\fra_{\bullet})$ is finitely generated, we can choose $m_1,\ldots,m_r\in S$ such that over suitable subsets in a finite affine open cover
of $X$,
$R(\fra_{\bullet})$ is generated over $\cO_X$ by elements in degrees in $\{m_1,\ldots,m_r\}$. We may and will assume that $m_i\in T$ for all $i$,
so $T$ is generated by $m_1,\ldots,m_r$. Therefore the saturation $S_+$ of $T$ is finitely generated. Note that if $\Delta_0$ satisfies the condition
in the proposition for $(\fra_m)_{m\in S_+}$, then we may take $\Delta$ to be any smooth fan with support $C$ with the property that every cone of $\Delta_0$ is a union of cones 
of $\Delta$. Indeed, the condition (\ref{eq_main_result}) holds trivially on the cones not contained in the support of $\Delta_0$. We thus may and will assume that
$S=S_+$. 

Since $R(\fra_{\bullet})$ is finitely generated, for every $m\in S$, the $\cO_X$-algebra $\bigoplus_{\ell\geq 0}\fra_{\ell m}$ is finitely generated
(see \cite[Lemma~4.8]{ELMNP}). In this case it follows from \cite[Chap. III, Section 1, Proposition 3]{Bourbaki} that there is a positive integer $d$ such that
$\fra_{d\ell m}=\fra_{dm}^{\ell}$ for all $\ell\geq 1$. We denote the smallest such $d$ by $d_m$. 

In what follows, it is convenient to use the formalism of asymptotic multiplicities, as in \cite{ELMNP}. If $v$ is a discrete valuation of the function field of $X$,
having center on $X$, and if $m\in S$, then we put
$$v^{\fra_{\bullet}}(m):=\inf_{\ell}\frac{v(\fra_{\ell m})}{\ell}=\lim_{\ell\to\infty}\frac{v(\fra_{\ell m})}{\ell},$$
where both the infimum and the limit are over those $\ell$ such that $\fra_{\ell m}\neq 0$. Note that by definition of $d_m$, 
we have $v^{\fra_{\bullet}}(m)=\tfrac{v(\fra_{\ell d_mm})}{\ell d_m}$ for every $m\in S$ and every $\ell\in {\mathbf Z}_{>0}$.

Our choice of $m_1,\ldots,m_r$ implies that for every 
$m\in S$, we have
\begin{equation}\label{eq_sum}
\fra_m=\sum_{\ell_1,\ldots,\ell_r}\fra_{m_1}^{\ell_1}\cdots\fra_{m_r}^{\ell_r},
\end{equation}
where the sum is over all $\ell_1,\ldots,\ell_r\in {\mathbf Z}_{\geq 0}$ with $m=\sum_{i=1}^r\ell_im_i$.
We now show that for every $m\in S$, we have
\begin{equation}\label{eq_claim}
v^{\fra_{\bullet}}(m)=\inf\left\{\sum_{i=1}^r\lambda_i\cdot v(\fra_{m_i})\mid \lambda_1,\ldots,\lambda_r\in {\mathbf Q}_{\geq 0}, m=\sum_{i=1}^r\lambda_im_i\right\}.
\end{equation}
In order to prove ``$\leq$", note that given  $\lambda_1,\ldots,\lambda_r\in {\mathbf Q}_{\geq 0}$ with  $m=\sum_{i=1}^r\lambda_im_i$,
we may choose $\ell\in {\mathbf Z}_{>0}$ such that $\ell\lambda_i\in {\mathbf Z}$ for all $i$. In this case the inclusion
$\prod_i\fra_{m_i}^{\ell\lambda_i}\subseteq\fra_{\ell m}$ implies
$$v(\fra_{\ell m})\leq \sum_{i=1}^r \ell\lambda_i\cdot v(\fra_{m_i})$$
and thus 
$$v^{\fra_{\bullet}}(m)\leq \frac{v(\fra_{\ell m})}{\ell}\leq \sum_{i=1}^r \lambda_i\cdot v(\fra_{m_i}).$$
This gives the inequality ``$\leq$" in 
 (\ref{eq_claim}). 
In order to prove the opposite inequality, note that if $\ell\in {\mathbf Z}_{>0}$ is such that $\fra_{\ell m}\neq 0$, then it follows from (\ref{eq_sum}) that there are
$\ell_1,\ldots,\ell_r\in {\mathbf Z}_{\geq 0}$ such that $\sum_i\ell_im_i=\ell m$ and 
$$v(\fra_{\ell m})\geq \sum_{i=1}^r\ell_i\cdot v(\fra_{m_i}).$$
Dividing by $\ell$ and then letting $\ell$ vary, we obtain the inequality ``$\geq$" in (\ref{eq_claim}). 

It follows from (\ref{eq_claim}) that we may apply Lemma~\ref{main_lemma} to obtain a fan $\Delta$ (that we may assume to be smooth), with support $C$,
such that for every valuation $v$ as above, we have 
$v^{\fra_{\bullet}}(m+m')=v^{\fra_{\bullet}}(m)+v^{\fra_{\bullet}}(m')$ whenever $m,m'\in S$ lie in the same cone of $\Delta$. 
Let $d$ be the least common multiple of the $d_w$, when $w$ runs over the primitive ray generators of $\Delta$. In this case, if $\sigma$ is a cone in 
$\Delta$ with primitive ray generators $e_1,\ldots,e_s$, then for every $p_1,\ldots,p_s\in {\mathbf Z}_{\geq 0}$, if $m=\sum_{i=1}^sp_ie_i$, then 
\begin{equation}\label{eq_final}
v(\fra_{dm})\leq \sum_{i=1}^sp_i\cdot v(\fra_{de_i})=\sum_{i=1}^sp_i\cdot v^{\fra_{\bullet}}(de_i)=v^{\fra_{\bullet}}(dm)\leq v(\fra_{dm}),
\end{equation}
where the first inequality follows from the inclusion $\prod_i\fra_{de_i}^{p_i}\subseteq\fra_m$. Therefore all inequalities in (\ref{eq_final}) are equalities.
Since 
$$v(\fra_{dm})=v\big(\fra_{de_1}^{p_1}\cdots \fra_{de_s}^{p_s}\big)$$
for every discrete valuation $v$ of the function field of $X$ that has center on $X$, it follows from \cite[Proposition~6.8.2]{HS} that
$$\overline{\fra_{dm}}=\overline{\fra_{de_1}^{p_1}\cdots \fra_{de_s}^{p_s}}.$$
This completes the proof.
\end{proof}

\subsection*{Acknowledgments}
We would like to thank Vlad Lazi\'{c} for
motivating us to write
this erratum and for several comments on preliminary versions.
We are also indebted to Sasha Barvinok for his comments in connection with 
Lemma~\ref{main_lemma}, 
especially for the suggestion to use the Duality Theorem in Linear Programming in the proof.


\section*{References}
\begin{biblist}

\bib{Barvinok}{book}{
   author={Barvinok, A.},
   title={A course in convexity},
   series={Graduate Studies in Mathematics},
   volume={54},
   publisher={American Mathematical Society, Providence, RI},
   date={2002},
}

\bib{Bourbaki}{book}{
   author={Bourbaki, N.},
   title={\'{E}l\'{e}ments de math\'{e}matique. Fascicule XXVIII. Algebre
commutative. Chapitre 3: Graduations, filtrations et topologies. Chapitre
   4:
 Id\'{e}aux premiers associ\'{e}s et d\'{e}composition primaire},
   series={},
   publisher={Hermann, Paris, },
   date={1961},
}

\bib{ELMNP}{article}{
   author={Ein, L.},
   author={Lazarsfeld, R.},
   author={Musta\c{t}\u{a}, M.},
   author={Nakamaye, M.},
   author={Popa, M.},
   title={Asymptotic invariants of base loci},
   journal={Ann. Inst. Fourier (Grenoble)},
   volume={56},
   date={2006},
   number={6},
   pages={1701--1734},
}

\bib{Ewald}{book}{
 author={Ewald, G.},
  title={Combinatorial convexity and algebraic geometry},
  series={Graduate Texts in Mathematics},
   volume={168},
   publisher={Springer-Verlag, New York},
   date={1996},
}

\bib{HS}{book}{
   author={Huneke, C.},
   author={Swanson, I.},
   title={Integral closure of ideals, rings, and modules},
   series={London Mathematical Society Lecture Note Series},
   volume={336},
   publisher={Cambridge University Press, Cambridge},
   date={2006},
}

\bib{Ziegler}{book}{
   author={Ziegler, G.~M.},
   title={Lectures on polytopes},
   series={Graduate Texts in Mathematics},
   volume={152},
   publisher={Springer-Verlag, New York},
   date={1995},
}

\end{biblist}
\end{document}